\documentclass[12pt]{extarticle}

\usepackage[margin=1.5in]{geometry}  
\usepackage{graphicx}              
\usepackage{amsmath}               
\usepackage{amsfonts}              
\usepackage{amsthm}                
\usepackage{framed}
\usepackage{enumitem}
\usepackage{comment}

\newtheorem{thm}{Theorem}
\newtheorem{lem}[thm]{Lemma}

\begin{document}

\nocite{*}

\title{\bf On the Number of Divisors of $n^2 -1$}

\author{\textsc{Adrian W. Dudek} \\ 
Mathematical Sciences Institute \\
The Australian National University \\ 
\texttt{adrian.dudek@anu.edu.au}}
\date{}

\maketitle

\begin{abstract}
We prove an asymptotic formula for the sum \(\sum_{n \leq N} d(n^2 - 1)\), where $d(n)$ denotes the number of divisors of $n$. During the course of our proof, we also furnish an asymptotic formula for the sum \(\sum_{d \leq N} g(d)\), where \(g(d)\) denotes the number of solutions $x$ in $\mathbb{Z}_d$ to the equation \(x^2 \equiv 1 \mod d\).
\end{abstract}

\section{Introduction}

It is the main purpose of this note to prove the following theorem.
\begin{thm} \label{main}
Let $d(n)$ denote the number of divisors of $n$. Then
$$\sum_{n \leq N} d(n^2-1) \sim \frac{6}{\pi^2} N \log^2 N$$
as $N \rightarrow \infty$. 
\end{thm}

In consideration of the more general sum $\sum_{n \leq N} d(n^2 +a)$, it was noted by Hooley \cite{hooley} that, in the case where $a = -k^2$, we may factorise $n^2+a$ as $(n-k)(n+k)$, and then the sum bears much resemblance to
\begin{equation} \label{inghamsum}
\sum_{n \leq N} d(n) d(n+2k),
\end{equation}
which was first studied by Ingham \cite{ingham}. As mentioned by Hooley, it is certainly possible in this case to compare these sums to show that
$$\sum_{n \leq N} d(n^2 -k^2) \sim C(k) N \log^2 N$$
as $N \rightarrow \infty$ for some constant $C(k)$. Elsholtz, Filipin and Fujita showed (see Lemma 3.5 of \cite{elsholtz}) that $C(1) \leq 2$. Trudgian \cite{trudgian} reduced this to $C(1) \leq 12/\pi^2$, before Cipu \cite{cipu} showed that $C(1) \leq 9/\pi^2$. Theorem \ref{main} of this note gives the result that $C(1) = 6/\pi^2$. 

However, rather than work from Ingham's asymptotic formula, we give a proof that requires information on the number of solutions to the equation $x^2 \equiv 1 \mod d$. Thus, before we prove Theorem \ref{main}, we first prove the following result which is of interest in its own right.

\begin{thm} \label{main2}
Let $g(d)$ denote the number of solutions to the equation $x^2 \equiv 1 \mod d$ such that $1 \leq x \leq d$. Then
$$\sum_{d < N} g(d) \sim \frac{6}{\pi^2} N \log N$$
as $N \rightarrow \infty$.
\end{thm}

After proving our two theorems, we give some insight on how one might generalise this work.

It should also be noted that the sum in Theorem \ref{main} plays a role in the theory of Diophantine $m$-tuples. We call a set of $m$ distinct integers $\{a_1, \ldots, a_m\}$ a Diophantine $m$-tuple if $a_i a_j+1$ is a perfect square for all $1 \leq i < j \leq m$. For example, the set $\{1,3,8,120\}$ is a Diophantine quadruple. It has been shown by Dujella \cite{dujella} that there are no diophantine $m$-tuples for $m \geq 6$, and it has been conjectured that there are no Diophantine quintuples, though this has yet to be proven. The best result in this direction is that of Trudgian \cite{trudgian}, who has recently shown that there are at most $2.3 \cdot 10^{29}$ Diophantine quintuples. In this context, the sum appearing in Theorem \ref{main} is useful, for it is equal to twice the number of Diophantine 2-tuples $\{a,b\}$ such that $ab + 1 \leq N^2$.

\section{Proof of main theorems}
We start by manipulating the divisor sum in the usual way. We have that
\begin{eqnarray*}
\sum_{n \leq N} d(n^2-1) & = & \sum_{n \leq N} \bigg( 2 \sum_{\substack{d | (n ^2 - 1) \\ d < n }} 1\bigg) \\
& = & 2 \sum_{d < N} \sum_{\substack{ d < n \leq N \\ n^2 \equiv 1 \text{ mod }d } }1
\end{eqnarray*}
where the inner sum is now over the integers $n$ in the interval $(d,N]$ such that $n^2$ is congruent to $1$ modulo $d$. We let $g(d)$ denote the number of solutions to the equation $x^2 \equiv 1 \mod d$ where $x \in \mathbb{Z}_d$. To estimate the inner sum, we first require the following lemma.

\begin{lem} \label{lem1}
Let $d$ be a positive integer. Writing $d = 2^a q$, where $q$ is odd and $a \geq 0$, it follows that $g(d) = 2^{\omega(q)+s(a)}$, where $\omega(q)$ denotes the number of distinct prime factors of $q$ and
 \begin{displaymath}
   s(a) = \left\{
     \begin{array}{lr}
       0 & \text{if }  a \leq 1\\
       1 & \text{if }   a = 2 \\
       2 & \text{if }  a \geq 3
     \end{array}
   \right.
\end{displaymath}
\end{lem}

\begin{proof}
This follows from Lemma 4.1 of Cipu \cite{cipu}.
\end{proof}

Denote by $Q(x,d)$ the number of positive integers $n \leq x$ such that $n^2 \equiv 1 \mod d$. Lemma \ref{lem1} allows us to estimate $Q(x,d)$, for in an interval of length $d$ there will be $g(d)$ such numbers that satisfy the congruence. Therefore, we have that
\begin{equation} \label{count}
Q(x,d) = g(d) \frac{x}{d} + O(g(d)).
\end{equation}
With this notation, we can write our original sum as
\begin{eqnarray*}
\sum_{n \leq N} d(n^2-1) & = &  2 \sum_{d < N} \bigg( Q(N,d) - Q(d,d) \bigg).
\end{eqnarray*}
It follows now from (\ref{count}) and the fact that $Q(d,d) = g(d)$ that
\begin{equation} \label{later}
\sum_{n \leq N} d(n^2-1)  =   2 N \sum_{d < N} \frac{g(d)}{d} + O \bigg( \sum_{d < N} g(d)\bigg).
\end{equation}
The order of the error term can be bounded in the straightforward way
$$\sum_{d < N} g(d) \ll \sum_{d < N} 2^{\omega(d)} \ll N \log N,$$
and so it remains to show that
$$ \sum_{d < N} \frac{g(d)}{d} \sim  \frac{3}{\pi^2} \log^2 N$$
as $N \rightarrow \infty$. To estimate this sum, we will use the following result, which can be found as Theorem 2.4.1 in Cojocaru and Murty \cite{cojocarumurty}.

\begin{lem} \label{murty}
Let
$$F(s) = \sum_{n=1}^{\infty} \frac{a_n}{n^s}$$
be a Dirichlet series with non-negative coefficients converging for $\text{Re}(s) >1$. Suppose that $F(s)$ extends analytically at all points on $\text{Re}(s)=1$ apart from $s=1$, and that at $s=1$ we can write
$$F(s) = \frac{H(s)}{(s-1)^{1-\alpha}}$$
for some $\alpha \in \mathbb{R}$ and some $H(s)$ holomorphic in the region $\text{Re}(s) \geq 1$ and non-zero there. Then
$$\sum_{n \leq x} a_n \sim \frac{c x}{( \log x)^{\alpha}}$$
with
$$c:=\frac{H(1)}{\Gamma(1 - \alpha)}$$
where $\Gamma$ is the Gamma function.
\end{lem}

This result allows one to step from some ``well-behaved'' Dirichlet series to an asymptotic formula for the partial sum of its coefficients. We will use this to prove Theorem \ref{main2}, by exploiting the multiplicity of the function $g(d)$ to construct an appropriate Dirichlet series.

\begin{proof}
We will consider the Dirichlet series
$$F(s) = \sum_{n=1}^{\infty} \frac{g(n)}{n^s}.$$
Note that as $g(n)$ is multiplicative, we have that
$$F(s) = \prod_p \bigg( 1 + \frac{g(p)}{p^s} + \frac{g(p^2)}{p^{2s}} + \cdots \bigg).$$
More specifically, from Lemma \ref{lem1} it follows that
$$F(s) = \bigg( 1 + \frac{1}{2^s} + \frac{2}{4^s} + 4 \bigg( \frac{1}{8^s} + \frac{1}{16^s} + \cdots \bigg) \bigg) \cdot \prod_{p \text{ odd}}  \bigg( 1 + \frac{2}{p^s} + \frac{2}{p^{2s}} + \cdots \bigg). $$
We now use the fact that
\begin{eqnarray*} 
\frac{\zeta^2 (s)}{\zeta(2s)} & = & \prod_p \frac{1 - p^{-2s}}{(1-p^{-s})^2} = \prod_p \frac{1+p^{-s}}{1-p^{-s}} \\
& = & \prod_{p} \bigg( 1 + \frac{2}{p^s} + \frac{2}{p^{2s}} + \cdots \bigg)
\end{eqnarray*}
where $\zeta(s)$ is the Riemann zeta-function (see Titchmarsh \cite{titchmarsh} for more details). Therefore, we have that
$$F(s) = \bigg( 1 + \frac{1}{2^s} + \frac{2}{4^s} +  \frac{4}{8^s - 4^s} \bigg) \bigg(\frac{1-2^{-s}}{1+2^{-s}} \bigg) \frac{\zeta^2(s)}{\zeta(2s)}.$$
By the properties of the Riemann zeta-function, $F(s)$ satisfies the conditions of Lemma \ref{murty} with $\alpha = -1$ and so we have that
$$\sum_{d < N} g(d) \sim c N \log N$$
where
$$c := \lim_{s \rightarrow 1} (s-1)^2 F(s) = \frac{1}{\zeta(2)} = \frac{6}{\pi^2}.$$ 
This completes the proof.
\end{proof}

Now, it follows by partial summation that
\begin{eqnarray*} 
\sum_{d < N} \frac{g(d)}{d} & = & \frac{6}{\pi^2} \int_1^N \frac{\log t}{t} \ dt + o\bigg(\int_1^N \frac{\log t}{t} \ dt \bigg) \\
& = & \frac{3}{\pi^2} \log^2 N + o(\log^2 N).
\end{eqnarray*}
Finally, an application of the above estimate into (\ref{later}) finishes the proof of Theorem \ref{main}.

\section{Further notes}

It would be interesting to see if one could extend this work so as to determine asymptotic estimates for the sums
$$\sum_{n \leq N} d(n^2 - r^2)$$
and
$$\sum_{d < N} g_r(d),$$
where $g_r(d)$ denotes the number of solutions of the equation $x^2 \equiv r^2 \mod d$ such that $1 \leq x \leq d$. If $r$ is fixed, then note that if $p$ is an odd prime and $k \geq 1$, the equation $x^2 \equiv r^2 \mod p^k$ yields
$$p^k | (x-r)(x+r).$$
For a sufficiently large prime $p$, there will be exactly two solutions to the above, namely $x = r$ and $x = p^k - r$. Therefore, we have $g_r(p^k) = 2$ for all sufficiently large primes $p$, and thus one will inevitably require the factor $\zeta^2(s)/\zeta(2s)$ in the construction of an appropriate Dirichlet series. Thus, one can expect to obtain asymptotics of the form
$$\sum_{n \leq N} d(n^2 - r^2) \sim \frac{A(r)}{\pi^2} N \log^2 N$$
and
$$\sum_{d < N} g_r(d) \sim \frac{B(r)}{\pi^2} N \log N,$$
where $A(r)$ and $B(r)$ are rational numbers dependent on $r$.

\section*{Acknowledgements}

The author is gracious of the financial support provided by an Australian Postgraduate Award and an ANU Supplementary Scholarship. He would also like to thank Dr Timothy Trudgian for introducing him to this problem, and for many conversations of an enjoyable nature.

\clearpage

\bibliographystyle{plain}

\bibliography{biblio}

\end{document}